\documentclass[11pt]{amsart}
\usepackage{enumerate, amssymb}

\usepackage{amsfonts}
\usepackage{enumerate}
\usepackage{amsmath}
\usepackage{amssymb}
\usepackage{graphicx}

\usepackage{enumerate}
\usepackage{amsmath}
\usepackage{amssymb}

\usepackage[latin1]{inputenc}

\usepackage[colorlinks, bookmarks=true]{hyperref}
\usepackage{color,graphicx,shortvrb}
\usepackage[active]{srcltx} 

\usepackage{amsmath}

\newtheorem{theorem}{Theorem}[section]

\newtheorem{lemma}[theorem]{Lemma}

\theoremstyle{definition}
\newtheorem{definition}[theorem]{Definition}

\theoremstyle{remark}
\newtheorem{remark}[theorem]{Remark}

\numberwithin{equation}{section}
\def\sqr#1#2{{\,\vcenter{\vbox{\hrule height.#2pt\hbox{\vrule width.#2pt
height#1pt \kern#1pt\vrule width.#2pt}\hrule height.#2pt}}\,}}

\def\Lip{\mathrm{Lip}}
\def\Iso{\mathrm{Iso}}
\def\Aut{\mathrm{Aut}}
\def\Int{\mathrm{Int}}
\def\diam{\mathrm{diam}}

\begin{document}

\title[$2$-local isometries on vector-valued Lipschitz spaces]{$2$-local standard isometries on vector-valued Lipschitz function spaces}

\author[A. Jim{\'e}nez-Vargas]{Antonio Jim{\'e}nez-Vargas}
\address[A. Jim{\'e}nez-Vargas]{Departamento de Matem{\'a}ticas, Universidad de Almer{\'i}a, 04120, Almer{\i}a, Spain}
\email{ajimenez@ual.es}

\author[L. Li]{Lei Li}
\address[L. Li]{School of Mathematical Sciences and LPMC, Nankai University, Tianjin 300071, China}
\email{leilee@nankai.edu.cn}

\author[A.M. Peralta]{Antonio M. Peralta}
\address[A.M. Peralta]{Departamento de An\'{a}lisis Matem\'{a}tico, Facultad de Ciencias, Universidad de Granada, 18071 Granada, Spain}
\email{aperalta@ugr.es}

\author[L. Wang]{Liguang Wang}
\address[L. Wang]{School of Mathematical Sciences, Qufu Normal University, Qufu 273165, China}
\email{wangliguang0510@163.com}

\author[Y.-S. Wang]{Ya-Shu Wang}
\address[Y.-S. Wang]{Department of Applied Mathematics, National Chung Hsing University, Taichung 402, Taiwan}
\email{yashu@nchu.edu.tw}

\date{\today}


\begin{abstract}
Under the right conditions on a compact metric space $X$ and on a Banach space $E$, we give a description of the $2$-local (standard) isometries on the Banach space $\Lip(X,E)$ of vector-valued Lipschitz functions from $X$ to $E$ in terms of a generalized composition operator, and we study when every $2$-local (standard) isometry on $\Lip(X,E)$ is both linear and surjective.
\end{abstract}

\maketitle

\section{Introduction}

Let $E$ and $F$ be Banach spaces. Let $\mathcal{S}$ be a subset of the space $L(E,F)$ of all continuous linear maps from $E$ into $F$. We shall follow the notation in \cite{NiPe2014,NiPe2015,CaPe2015,CaPe2017} and \cite{Semrl97}. Accordingly to those references, a (non-necessarily linear nor continuous) mapping $\Delta\colon E\to F$ is a \emph{2-local $\mathcal{S}$-map} if for any $x,y\in E$, there exists $T_{x,y}\in\mathcal{S}$, depending on $x$ and $y$, such that $$\hbox{$\Delta(x)=T_{x,y}(x)$ and $\Delta(y)=T_{x,y}(y)$.}$$

The most studied class of $2$-local $\mathcal{S}$-maps is that related to the set $\mathcal{S}=\Iso(E,F)$ of all surjective linear isometries from $E$ onto $F$. We shall write $\Iso(E)$ instead of $\Iso(E,E)$. A Banach space $E$ is said to be \emph{$2$-iso-reflexive} if every $2$-local $\Iso(E)$-map (also called \emph{$2$-local isometry}) on $E$ is both linear and surjective.\smallskip

Similarly, when $A$ is a Banach algebra and $\Aut(A)$ denotes the set of all automorphisms on $A$, we shall say that $A$ is \textit{$2$-auto-reflexive} if every $2$-local $\Aut(A)$-map (also named \emph{$2$-local automorphism}) on $A$ is both linear and surjective.\smallskip

A problem addressed in the literature by different authors is to study when a Banach space is $2$-iso-reflexive or a Banach algebra is $2$-auto-reflexive. Another interesting problem is to determine when a $2$-local isometry or a $2$-local automorphism is linear. Let us present an illustrative list of examples.\smallskip

\begin{enumerate}[$(\checkmark)$]
\item For $1\leq p<\infty$ and $p\neq 2$, the space $\ell^p$ is $2$-iso-reflexive (Al-Halees and Fleming \cite{AF09});
\item For a Hilbert space $H$, the algebra $B(H)$ is $2$-auto-reflexive (\v{S}emrl \cite{Semrl97} for the case in which $H$ is infinite dimensional and separable, and Ayupov and Kudaybergenov \cite{AyuKuday2012} for the general case);
\item For an infinite dimensional separable complex Hilbert space $H$, the C$^*$-algebra $B(H)$ is $2$-iso-reflexive, and every C$^*$-subalgebra $A$ of $B(H)$ which is an extension of the algebra
of all compact operators by a separable commutative C$^*$-algebra is $2$-iso-reflexive (Moln\'{a}r \cite{Mol02});
\item Every $2$-local $^*$-homomorphism from a von Neumann algebra (respectively, from a compact C$^*$-algebra) into a C$^*$-algebra is a linear $^*$-homomorphism (Burgos, Fern{\' a}ndez-Polo, Garc{\'e}s, and Peralta \cite{BurFerGarPe2015}).
\item For a first countable compact Hausdorff space $X$, the C$^*$-algebra $C(X)$ of all complex-valued continuous functions on $X$ is $2$-auto-reflexive (Moln\'{a}r \cite{Mol02b});
\item Suppose $L$ is a first countable and $\sigma$-compact locally compact Hausdorff space. Then the C$^*$-algebra $C_0(L)$ of all complex-valued continuous functions on $L$ which vanish at infinity, is $2$-iso-reflexive (Gy\H{o}ry \cite{Gyo01}). Gy\H{o}ry also showed the existence of an uncountable discrete space $L$ for which there exists a non-surjective $2$-local automorphism on $C_0(L)$;
\item If $A$ is a uniform algebra on a compact Hausdorff space $X$, that is, a closed subalgebra of $C(X)$ which separates the points of $X$ and contains the constant functions, then every $2$-local automorphism $T$ on $A$ is a linear isometry from $A$ onto $T(A)$.\smallskip

Given a compact subset $K\subset\mathbb{C}$, let $A(K)$ denote the algebra of all complex-valued continuous functions on $K$ which are holomorphic on the interior $\Int(K)$ of $K$. If $K$ is a connected compact subset of $\mathbb{C}$ such that $\Int(K)$ has finitely many components and $\overline{\Int(K)} = K$, then every $2$-local isometry ($2$-local automorphism) $T$ on $A(K)$ is a surjective linear isometry (respectively, automorphism). The same conclusion holds when $K$ is the closure of a strictly pseudoconvex domain in $\mathbb{C}^2$ with boundary of class $C^2$ (Hatori, Miura, Oka and Takagi \cite{HaMiOkTak07});\smallskip

\item If $K$ is a $\sigma$-compact metric space and $E$ is a smooth reflexive Banach space, then $C_0(K,E)$ is $2$-iso-reflexive if and only if $E$ is $2$-iso-reflexive (Al-Halees and Fleming \cite{AF09});
\item Every weak-$2$-local isometry between uniform algebras is linear  (Li, Peralta, Wang and Wang \cite{LiPeWangWang17}).
\end{enumerate}

2-local derivations on C$^*$-algebras have been studied in \cite{Semrl97,AyuKuday2014,AyuKudPe2014,NiPe2014, NiPe2015,CaPe2015,CaPe2017,JorPe} and \cite{JorPe17}.\smallskip

Let $X$ be a metric space and let $E$ be a normed space. Throughout this paper the \textit{Lipschitz space} $\Lip(X,E)$ is the space of all bounded Lipschitz functions $f$ from $X$ to $E$, endowed with the norm
$$
\left\|f\right\|_L=\max\left\{L(f),\left\|f\right\|_\infty\right\},
$$
where $L(f)$ denotes the Lipschitz constant of $f$ and $\left\|f\right\|_\infty$ the uniform norm of $f$. In the case $E=\mathbb{C}$, we shall simply write $\Lip(X)$. It is usual to consider also the vector space $\Lip(X)$ equipped with the norm $\left\|f\right\|_s=L(f)+\left\|f\right\|_\infty$ for all $f\in\Lip(X)$.\smallskip

The aim of this note is to study $2$-local isometries between spaces of vector-valued Lipschitz functions $\Lip(X,E)$. This is probably the first case where a vector-valued study is considered. \smallskip

The complex-valued case was treated in \cite{JV2008,LiPeWangWang17}. Namely, if the set
$$
\mathcal{S}:=\Iso((\Lip(X),\|.\|),(\Lip(Y),\|.\|))
$$
is canonical (in the sense that every element of $\mathcal{S}$ can be expressed as a weighted composition operator $\tau C_\psi$, where the weight $\tau$ is a unimodular constant and the symbol $\psi$ is an isometry from $Y$ onto $X$) and $\|.\|$ stands indistinctly for the norms $\|\cdot\|_L$ and $\|\cdot\|_s$, it was stated in \cite{LiPeWangWang17} that every weak-$2$-local $\mathcal{S}$-map from $\Lip(X)$ to $\Lip(Y)$ is linear and, as a consequence, the space $(\Lip(K), \|.\|_s)$ is $2$-iso-reflexive for each compact metric space $K$ such that $\Iso(\Lip(K),\|.\|_{s})$ is canonical.\smallskip

In order to study $2$-local isometries, an appropriate starting point is a good knowledge on the surjective linear isometries of the concrete space to be studied. The key description to begin our study is provided by results due to the first author and Villegas-Vallecillos in \cite{JV2008} and Botelho, Fleming and Jamison in \cite{BFJ11}, and more concretely from the next theorem borrowed from the just quoted paper.

\begin{theorem}\label{thm BoFleJam}\cite[Theorem 10]{BFJ11}
Suppose that $X$, $Y$ are compact and pathwise connected metric spaces, and $E$, $F$ are smooth, reflexive Banach spaces. Then, for each surjective linear isometry $T\colon \Lip(X, E)\to \Lip(Y, F)$,  there exist a Lipschitz homeomorphism $\varphi\colon Y\to X$ with $L(\varphi)\leq \max\{1, \diam(X)\}$ and $L(\varphi^{-1})\leq \max\{1, \diam(Y)\},$ and a Lipschitz map $V\colon Y\to\Iso(E,F)$ such that
$$
T(f)(y)=V(y)(f(\varphi(y))),\quad\forall y\in Y,\ f\in\Lip(X, E).
$$
\end{theorem}$\hfill\Box$

The previous theorem will be frequently applied without any explicit mention. Araujo and Dubarbie \cite{ArDubar2011} introduced an alternative subclass of surjective linear isometries admitting an expression similar to that of Theorem \ref{thm BoFleJam}, which is also appropriate for our purposes.\smallskip

Let $X$ and $Y$ be metric spaces and let $E$ and $F$ be normed spaces. A bijective map $\varphi\colon Y\to X$ \emph{preserves distances less than $2$} if $d(\varphi(x),\varphi(y))=d(x, y)$ whenever $d(x,y)<2$. The symbol $\Iso_{<2}(Y,X)$ will stand for the set of all bijective maps $\varphi\colon Y\to X$ such that both $\varphi$ and $\varphi^{-1}$ preserve distances less than 2. Notice that every element $\varphi\in\Iso_{<2}(Y,X)$ is a Lipschitz homeomorphism when $X$ and $Y$ are bounded.

\begin{definition}\cite[Definition 2.3]{ArDubar2011}
A map $T\colon\Lip(X,E)\to\Lip(Y,F)$ is a \emph{standard isometry} if there exist $\varphi\in \Iso_{<2}(Y,X)$ and a map $J\colon Y\to\Iso(E,F)$ constant on each $2$-component of $Y$ such that
$$
T(f)(y)=J(y)(f(\varphi(y)))
$$
for all $f\in\Lip(X,E)$ and $y\in Y$. Under these conditions, it can be easily seen that the map $J$ is actually a Lipschitz mapping (compare \cite[page 827, lines 28--31]{BFJ11}).
\end{definition}

Notice that every standard isometry is a surjective linear isometry. Theorem 3.1 in \cite{ArDubar2011} provides a condition under which both types of isometries coincide.\smallskip

Let ${\mathbf{S}}$ denote the set of all standard isometries from $\Lip(X,E)$ to $\Lip(Y,F)$. A mapping $\Delta\colon\Lip(X,E)\to\Lip(Y,F)$ which is a $2$-local $\mathbf{S}$-map will be called a \emph{$2$-local standard isometry}. We shall say that $\Lip(X,E)$ is \emph{$2$-standard-iso-reflexive} if every $2$-local standard isometry on $\Lip(X,E)$ is both linear and surjective.\smallskip

Under certain conditions on a compact metric space $X$ and on a normed space $E$, we shall give in this note a description of all $2$-local isometries and all $2$-local standard isometries on $\Lip(X,E)$ by means of a generalized composition operator, and we shall study when the space $\Lip(X,E)$ is $2$-iso-reflexive and $2$-standard-iso-reflexive, respectively.

\section{2-local isometries on vector-valued Lipschitz functions}

We shall make use of the following auxiliary functions. Let $(X,d)$ be a metric space and pick $x\in X$. It is immediate that $h_x\colon X\to [0,1]$, defined by
$$
h_x(z)=\max\left\{0, 1-d(z,x)\right\},\quad \forall z\in X,
$$
lies in $\Lip(X)$ and $\left\|h_x\right\|_L=1$. In fact, $h_x^{-1}(\{1\})=\{x\}$. If $X$ is bounded, the function $f_x\colon X\to\mathbb{R}_0^+$, given by
$$
f_x(z)=d(z,x),\quad \forall z\in X,
$$
is also in $\Lip(X)$ and $\left\|f_x\right\|_L\leq\max\left\{1,\diam(X)\right\}$.\smallskip

Given $f\in \Lip(X)$ and $e\in E$, the symbol $f\otimes e$ will denote the function defined by $f\otimes e(x)=f(x)e$ for all $x\in X$. Notice that $f\otimes e$ belongs to $\Lip(X,E)$ and $\left\|f\otimes e\right\|_L=\left\|f\right\|_L\left\|e\right\|$.\smallskip

For every normed space $E$, we write $B_E$ and $S_E$ for its closed unit ball and its unit sphere, respectively. Let us recall that a normed space $E$ is \emph{smooth} if for any point $e\in S_{E}$, there exists a unique functional $e^*\in S_{E^*}$ such that $e^*(e)=1$. It is known that a space $E$ is smooth if and only if its norm has a G\^{a}teaux differential at all points in $S_{E}$.\smallskip

Let $X$ and $Y$ be metric spaces, let $E$ and $F$ be normed spaces and let $\Delta\colon\Lip(X,E)\to\Lip(Y,F)$ be a mapping. For any $x\in X$, $u^*\in B_{E^*}$ and $f\in\Lip(X, E)$, we define the sets
$$
\mathcal{A}_{x,u^*,f}=\left\{(y, v^*)\in Y\times B_{F^*}\colon  v^*(\Delta (f)(y))=u^*(f(x))\right\}
$$
and
$$
\mathcal{A}_{x, u^*}=\bigcap_{f\in \Lip(X, E)}\mathcal{A}_{x,u^*, f}.
$$

\begin{lemma}\label{lem:vL1}
Let $X$ and $Y$ be compact metric spaces {\rm(}respectively, pathwise connected compact metric spaces{\rm)}, and let $E$ and $F$ be smooth reflexive Banach spaces. Then, for each $2$-local standard isometry {\rm(}respectively, for every $2$-local isometry{\rm)} $\Delta\colon\Lip(X,E)\to\Lip(Y,F)$, the set $\mathcal{A}_{x,u^*}$ is not empty for every $x\in X$ and $u^*\in S_{E^*}$. In such a case, $v^*\in S_{F^*}$ for every $(y,v^*)\in\mathcal{A}_{x, u^*}$.
\end{lemma}

\begin{proof} We shall only prove the statement concerning $2$-local standard isometries. The other statement follows with similar arguments and Theorem \ref{thm BoFleJam}.\smallskip

Assume that $\Delta$ is a $2$-local standard isometry, and let us fix $x\in X$ and $u^*\in S_{E^*}$.\smallskip

We shall first prove that $\mathcal{A}_{x,u^*,f}$ is not empty for every $f\in\Lip(X,E)$. Namely, there exists a standard isometry $T_f\colon\Lip(X,E)\to\Lip(Y,F)$ satisfying $\Delta(f)=T_f(f)$, and hence there exist $\varphi_f\in \Iso_{<2}(Y,X)$ and a Lipschitz mapping $J_f\colon Y\to\Iso(E,F)$ constant on each $2$-component of $Y$ such that
$$
T_f(h)(y)=J_f(y)(h(\varphi_f(y)))
$$
for all $h\in\Lip(X, E)$ and $y\in Y$. We can choose $y_0\in Y$ such that $\varphi_f(y_0)=x$, and thus
$$
\Delta(f)(y_0)=J_f(y_0)(f(\varphi_f(y_0)))=J_f(y_0)(f(x)).
$$
Taking $v_0^*=u^*J_f (y_0)^{-1}\in B_{F^*}$, we have $v_0^*(\Delta(f)(y_0))=u^*(f(x))$, and hence $(y_0,v_0^*)\in \mathcal{A}_{x,u^*,f}$, as required.\smallskip

Since $F$ is reflexive, the set $B_{F^*}$ is compact when equipped with the weak topology. The space $Y\times B_{F^*}$ will be regarded as a compact space with the product topology of both compact spaces. We claim that $\mathcal{A}_{x,u^*,f}$ is a closed subset of $Y\times B_{F^*}$ for every $f\in \Lip(X,E)$. To prove the claim, assume that $(y_\lambda, v^*_\lambda)$ is a net in $\mathcal{A}_{x,u^*, f}$ converging to $(y,v^*)$ in $Y\times B_{F^*}$, that is, $y_\lambda\to y$ in $Y$ and $v^*_\lambda\to v^*$ weakly. For each $\lambda$, we have $v_\lambda^*(\Delta (f)(y_\lambda))=u^*(f(x))$. Since $\Delta(f)\in \Lip(Y,F)$, it follows that $\Delta(f)(y_\lambda)\to \Delta(f) (y)$ in the norm topology of $F$. Since $v^*_\lambda\to v^*$ weakly, it can be easily seen that $v^*_\lambda(\Delta (f)(y_\lambda))$ converges to $v^*(\Delta (f)(y))$, and thus $v^*(\Delta(f)(y))=u^*(f(x))$, which proves the claim.\smallskip

Let us observe that the arguments above are valid to prove that for each $x\in X$, $u^*\in B_{E^*}$ and $f\in \Lip(X,E)$ the set $\mathcal{A}_{x,u^*,f}$ is non-empty and closed in $Y\times B_{F^*}.$\smallskip

We return to the case in which $u^*\in S_{E^*}$. Since $Y\times B_{F^*}$ is compact, in order to show that the set $\mathcal{A}_{x,u^*}$ is not empty it suffices to show that the family
$$
\left\{\mathcal{A}_{x,u^*,f}\colon  f\in \Lip(X,E)\right\}
$$
has the finite intersection property. Let $f_1,f_2,\ldots,f_n\in\Lip(X,E)$ and define $g\colon X\to\mathbb{R}$ by
$$
g(z)=1+\sum_{i=1}^n\left(2\|f_i\|_L-\|f_i(z)-f_i(x)\|\right) \quad (z\in X).
$$
It can be easily checked that $g\in\Lip(X,\mathbb{R}^+)$ and $f=h_xg\in\Lip(X,\mathbb{R}_0^+)$. Recall that $h_x$ is in $\Lip(X,[0,1])$ with $\left\|h_x\right\|_L=1$ and satisfies that $h_x^{-1}(\{1\})=\{x\}$. Since $E$ is reflexive, we can take $e\in S_E$ such that $u^*(e)=1$. Given $1\leq i\leq n$, since $\Delta$ is a $2$-local standard isometry, there exists a standard isometry $T_{f\otimes e, f_i}\colon\Lip(X,E)\to\Lip(Y,F)$ such that $\Delta (f\otimes e)=T_{f\otimes e, f_i}(f\otimes e)$ and $\Delta(f_i)=T_{f\otimes e,f_i}(f_i)$.
This implies the existence of $\varphi_{f\otimes e,f_i}\in \Iso_{<2}(Y,X)$ and a Lipschitz map $y\mapsto J_{f\otimes e, f_i}(y)$ from $Y$ into $\Iso(E,F)$ such that
$$
T_{f\otimes e, f_i}(h)(y)=J_{f\otimes e, f_i}(y)(h(\varphi_{f\otimes e,f_i}(y))),\quad\forall y\in Y,\ h\in\Lip(X,E).
$$
For each $(y,v^*)\in \mathcal{A}_{x,u^*,f\otimes e}$ and $1\leq i\leq n$, one can derive that
\begin{align*}
0<h_x(x)g(x)=f(x)&=u^*((f\otimes e)(x))=v^*(\Delta (f\otimes e)(y))\\
                 &=v^*(J_{f\otimes e, f_i}(y)((f\otimes e)(\varphi_{f\otimes e, f_i}(y))))\\
                 &=f(\varphi_{f\otimes e, f_i}(y)) v^*(J_{f\otimes e, f_i}(y)(e))\\
                 &\leq f(\varphi_{f\otimes e,f_i}(y))\\
								 &=h_x(\varphi_{f\otimes e, f_i}(y))g(\varphi_{f\otimes e, f_i}(y)).
\end{align*}
Here, $v^*J_{f\otimes e, f_i}(y)$ belongs to $B_{E^*}$ as $J_{f\otimes e, f_i}(y)$ is a linear surjective isometry.\smallskip

Since $h_x$ and $g$ attain their maximum at $x$, we deduce that $h_xg$ attains its maximum at $x$ and $\varphi_{f\otimes e,f_i}(y)$. It follows that $h_x$ attains its maximum at $\varphi_{f\otimes e,f_i}(y)$,
and therefore $h_x(\varphi_{f\otimes e,f_i}(y))=h_x(x)=1$. The property $h_x^{-1}(\{1\})=\{x\}$ implies that $\varphi_{f\otimes e,f_i}(y)=x$.  In this case, we have $v^*(J_{f\otimes e, f_i}(y)(e))=1$, and hence $u^*=v^*(J_{f\otimes e, f_i}(y))$ because $E$ is smooth. Consequently, we have
$$
v^*\left(\Delta (f_i)(y)\right)=v^*\left(J_{f\otimes e,f_i}(y)(f_i(\varphi_{f\otimes e,f_i}(y)))\right)=u^*(f_i(x)).
$$
This means that $(y,v^*)\in\mathcal{A}_{x,u^*,f_i}$ for all $1\leq i \leq n$, and thus the family $\{\mathcal{A}_{x,u^*,f}\colon f\in\Lip(X, E)\}$ satisfies the finite intersection property.\smallskip

Concerning the last statement, if $u^*\in S_{E^*}$ and $(y,v^*)\in \mathcal{A}_{x, u^*}$, then $u^*(f(x))=v^*(\Delta(f)(y))$ for all $f\in \Lip(X,E)$. Clearly, $\|v^*\|\leq 1$. By the reflexivity of $E$, we can find $u\in S_E$ such that $u^*(u)=1$. Therefore, since $\Delta$ is an isometry, we get
$$
1=u^*(u)=u^*(h_x\otimes u(x))=v^*(\Delta(h_x\otimes u)(y))\leq\|v^*\|\left\|h_x\otimes u\right\|_L\leq\left\|v^*\right\|.
$$
\end{proof}

For each mapping $\Delta\colon\Lip(X,E)\to\Lip(Y,F)$, let us define
$$
\mathfrak{Y}=\left\{(y,v^*)\in Y\times S_{F^*} \colon (y,v^*)\in \mathcal{A}_{x,u^*} \hbox{ for some  } x\in X, \ u^*\in S_{E^*}\right\}.
$$
Let $\pi_1\colon Y\times F^*\to Y$ denote the projection of $Y\times F^*$ onto the first component. The symbol $Y_0$ will denote the set $\pi_1(\mathfrak{Y}).$

\begin{lemma}\label{lem:vL2}
Under the hypothesis of Lemma \ref{lem:vL1}, the following statements hold:
\begin{enumerate}[$(a)$]
\item Let $(y,v_1^*),(y,v_2^*)\in \mathfrak{Y}$,  $u_1^*,u_2^*\in S_{E^*}$ and $x_1,x_2\in X$. If $(y,v_1^*)\in \mathcal{A}_{x_1,u_1^*}$ and $(y,v_2^*)\in \mathcal{A}_{x_2,u_2^*}$, then $x_1=x_2$.
\item Let $(y,v^*)\in\mathfrak{Y}$, $u_1^*,u_2^*\in S_{E^*}$ and $x_0\in X$. If $(y,v^*)\in \mathcal{A}_{x_0,u_1^*}$ and $(y,v^*)\in \mathcal{A}_{x_0,u_2^*}$, then $u_1^*=u_2^*$.
\item We can define a surjective mapping $\psi$ from $Y_0$ to $X$ defined by the following property: for each $y\in Y_0,$ $\psi(y)$ is the unique element in $X$ such that for any $v^*\in S_{F^*}$ with $(y,v^*)\in \mathfrak{Y}$, there exists a unique $u^*\in S_{E^*}$  such that $(y,v^*)\in \mathcal{A}_{\psi(y),u^*}$.
\end{enumerate}
\end{lemma}

\begin{proof} We shall only prove the case in which $\Delta$ is a $2$-local standard isometry, the other statement follows with similar arguments and Theorem \ref{thm BoFleJam}.\smallskip

$(a)$ Assume that $(y,v_1^*)\in \mathcal{A}_{x_1,u_1^*}$ and $(y,v_2^*)\in \mathcal{A}_{x_2,u_2^*}$. Then, for each $j=1,2$, we have
$$
u^*_j(f(x_j))=v^*_j(\Delta(f)(y))
$$
for every $f\in\Lip(X,E)$. For $j=1,2$, by the reflexivity of $E$ choose $u_j\in S_E$ such that $u^*_j(u_j)=1$ and consider the function $f_j=h_{x_j}\otimes u_j$. It follows from the hypothesis on $\Delta$ that there exist $\varphi_{f_1,f_2}\in\Iso_{<2}(Y,X)$ and a Lipschitz map $y\mapsto J_{f_1,f_2}(y)$ from $Y$ into $\Iso(E,F)$ such that
$$
\Delta(f_j)(y)=J_{f_1, f_2}(y)(f_j(\varphi_{f_1,f_2}(y)))\quad (j=1,2).
$$
Therefore
$$
1=u^*_j(f_j(x_j))=v^*_j(\Delta(f_j)(y))=h_{x_j}(\varphi_{f_1,f_2}(y))v^*_j(J_{f_1,f_2}(y)(u_j))\quad (j=1,2),
$$
which implies that $v^*_j(J_{f_1,f_2}(y)(u_j))=1$ for $j=1,2$ and hence $h_{x_j}(\varphi_{f_1,f_2}(y))=1$ for $j=1,2$. Since $h_{x_j}^{-1}(\{1\})=\{x_j\}$ for $j=1,2$, it follows that $x_2=\varphi_{f_1,f_2}(y)=x_1$, witnessing the first statement.\smallskip

$(b)$ Suppose $(y,v^*)\in \mathcal{A}_{x_0,u_1^*}$ and $(y,v^*)\in \mathcal{A}_{x_0,u_2^*}$. In this case the identities
$$
u^*_1(f(x_0))=v^*(\Delta(f)(y))=u^*_2(f(x_0))
$$
hold for every $f\in\Lip(X,E)$. In particular,
$$
u^*_1(e)=u^*_1(h_{x_0}\otimes e(x_0))=u^*_2(h_{x_0}\otimes e(x_0))=u^*_2(e)
$$
for all $e\in E$, and thus $u_1^*=u_2^*$.\smallskip

$(c)$ Let $Y_0=\pi_1(\mathfrak{Y})$ be the set of all $y\in Y$ such that $(y, v^*)\in \mathfrak{Y}$ for some $v^*\in S_{F^*}$. By (a) and (b), for each $y\in Y_0$, there exists a unique $\psi(y)\in X$ such that for each $(y,v^*)\in \mathfrak{Y}$, we have $(y,v^*)\in \mathcal{A}_{\psi(y),u^*}$ for a unique $u^*\in S_{E^*}$. The surjectivity of $\psi$ follows from Lemma \ref{lem:vL1}.
\end{proof}

We can establish now a Lipschitz version of \cite[Theorem 6]{AF09}.

\begin{theorem}\label{thm:vL3}
Suppose that $X$ and $Y$ are compact and pathwise connected metric spaces, and $E$ is a smooth and reflexive Banach space which is $2$-iso-reflexive.
Assume that $\Delta\colon\Lip(X,E)\to\Lip(Y,E)$ is a $2$-local isometry.
Then there exist a subset $Y_0$ of $Y$, a Lipschitz bijection $\psi$ from $Y_0$ onto $X$, and a mapping $y\in Y_0\mapsto V(y)\in\Iso(E)$, which is continuous from $Y_0$ into $\Iso(E)$ with the strong operator topology {\rm(}SOT\;{\rm)}, such that
$$
\Delta(f)(y)=V(y)(f(\psi(y))),\quad\forall y\in Y_0,\ f\in\Lip(X,E).
$$
\end{theorem}

\begin{proof}
Let $Y_0$ and $\psi\colon Y_0\to X$ be the set and the surjective mapping defined in Lemma \ref{lem:vL2}. We shall first show that $\psi$ is injective. To this end, let $y_1, y_2\in Y_0$ be such that $x_0=\psi(y_1)=\psi(y_2)\in X$. By definition, for each $j=1,2$ we have
$$
v_j^*(\Delta(f)(y_j))=u_j^*(f(x_0))
$$
for all $f\in\Lip(X,E)$, where $u_j^*,v_j^*\in S_{E^*}$ with $(y_j,v_j^*)\in\mathcal{A}_{x_0,u_j^*}$. Given $j\in\{1,2\}$, choose $u_j\in S_E$ such that $u_j^*(u_j)=1$ and take $f_j=h_{x_0}\otimes u_j$. Then we have
$$
v_j^*(\Delta (f_j)(y_j))=u_j^*(f_j(x_0))=u_j^*(u_j)=1\quad (j=1,2).
$$
Since $\Delta$ is a $2$-local isometry, there exists a linear isometry $\Delta_{f_1,f_2}$ from $\Lip(X,E)$ onto $\Lip(Y,E)$ which coincides with $\Delta$ at $f_1$ and $f_2$. By Theorem \ref{thm BoFleJam}, there are a Lipschitz homeomorphism $\varphi_{f_1,f_2}\colon Y\to X$ and a Lipschitz mapping $y\mapsto V_{f_1, f_2}(y)$ from $Y$ into $\Iso(E)$, such that
$$
\Delta(f_j)(y_j)=V_{f_1,f_2}(y_j)(f_j(\varphi_{f_1,f_2}(y_j)))=h_{x_0}(\varphi_{f_1,f_2}(y_j))V_{f_1,f_2}(y_j)(u_j),
$$ for all $j=1,2$. It follows that
$$
1=v_j^*(\Delta (f_j)(y_j))=h_{x_0}(\varphi_{f_1,f_2}(y_j))\ v_j^*(V_{f_1,f_2}(y_j)(u_j))\quad (j=1,2),
$$
which implies $v_j^*(V_{f_1,f_2}(y_j)(u_j))=1$, and $h_{x_0}(\varphi_{f_1,f_2}(y_j))=1$ for $j=1,2$. Since $h_{x_0}^{-1}(\{1\})=\{x_0\}$, we infer that $\varphi_{f_1,f_2}(y_j)=x_0$ for $j=1,2$. Since $\varphi_{f_1,f_2}$ is injective, we get $y_1=y_2$, which proves that $\psi$ is injective.\smallskip

For any $y\in Y_0,$ let us define a mapping $V(y) \colon  E\to E$ given by
$$
V(y)(e)=\Delta (f)(y),
$$
where $f$ is any function in $\Lip(X,E)$ satisfying $f(\psi(y))=e$. We first show that $V(y)$ is well-defined. In a first step, fix $e\in S_E$, take $g=h_{\psi(y)}\otimes e$ and let $f$ be any function in $\Lip(X,E)$ for which $f(\psi(y))=e$. Recall that $g$ is in $\Lip(X,E)$ with $||g||_L=1$ and satisfies that $g(\psi(y))=e$ and $\{z\in X: ||g(z)||=1\}=\{\psi(y)\}$. Since $\Delta$ is a $2$-local isometry, there exist a Lipschitz homeomorphism $\varphi_{f,g}\colon  Y\to X$ and a Lipschitz mapping $y\mapsto V_{f,g}(y)$ from $Y$ into $\Iso(E)$ such that
$$
\Delta(f)(y)=V_{f,g}(y)(f(\varphi_{f,g}(y)))
$$
and
$$
\Delta(g)(y)=V_{f,g}(y)(g(\varphi_{f,g}(y))).
$$
Let $e^*\in S_{E^*}$ such that $e^*(e)=\|e\|$. By Lemmas \ref{lem:vL1} and \ref{lem:vL2}, there exists $\tilde{e}^*\in S_{E^*}$ such that
$$
\tilde{e}^*(\Delta (g)(y))=e^*(g(\psi(y)))=e^*(e)=\|e\|=1.
$$
It follows that
$$
1=\tilde{e}^*(\Delta (g)(y))=\tilde{e}^*[V_{f,g}(y)(g(\varphi_{f,g}(y)))]\leq\|g(\varphi_{f,g}(y))\|\leq \|g\|_\infty=1.
$$
Since $\{z\in X: \|g(z)\|=1\}=\{\psi(y)\}$, we deduce that $\varphi_{f,g}(y)=\psi(y)$, and then
\begin{gather*}
\Delta (f)(y)=V_{f,g}(y)(f(\varphi_{f,g}(y)))=V_{f,g}(y)(f(\psi(y)))\\
=V_{f,g}(y)(g(\psi(y)))=V_{f,g}(y)(g(\varphi_{f,g}(y)))=\Delta(g)(y).
\end{gather*}
We have just proved that if $f_1,f_2$ are functions in $\Lip(X,E)$ satisfying $f_1(\psi(y))=f_2(\psi(y))=e$ for some $e\in S_E$, then
$$
\Delta(f_1)(y)=\Delta (f_2)(y).
$$
Now, let $e$ be a non-zero element in $E$. It is well known that every $2$-local isometry is 1-homogeneous, that is, $\Delta(\lambda f)=\lambda\Delta (f)$ for every $\lambda\in\mathbb{K}$ and $f\in\Lip(X,E)$ (compare \cite[Lemma 2.1]{NiPe2014} or \cite[Lemma 2.1]{CaPe2015}). Let us take $f_1,f_2$ in $\Lip(X, E)$ satisfying $f_1(\psi(y))=f_2(\psi(y))=e$. By the above-proved equality and the homogeneity of $\Delta$, we obtain
$$
\Delta(f_1)(y)=\|e\|\Delta\left(\frac{1}{\|e\|}f_1\right)(y)=\|e\|\Delta\left(\frac{1}{\|e\|}f_2\right)(y)=\Delta (f_2)(y).
$$
Finally, let us suppose that $e=0$. Since $y\in Y_0$, Lemmas \ref{lem:vL1} and \ref{lem:vL2} assure the existence of $u^*,v^*\in S_{E^*}$ such that
$$
v^*(\Delta(h)(y))=u^*(h(\psi(y)))
$$
for all $h\in\Lip(X,E)$. Let us take $u\in S_{E}$ with $u^*(u)=1$ and let $g=h_{\psi(y)}\otimes u$. In such a case, $v^*(\Delta(g)(y))=u^*(g(\psi (y)))=1$. Let $f$ be any function in $\Lip(X,E)$ for which $f(\psi(y))=0$. Applying that $\Delta$ is a $2$-local isometry, we deduce the existence of a Lipschitz homeomorphism $\varphi_{f,g}\colon Y\to X$ and a Lipschitz mapping $y\mapsto V_{f,g}(y)$ from $Y$ into $\Iso(E)$, such that
$$
\Delta (f)(y)=V_{f,g}(y)(f(\varphi_{f,g}(y))
$$
and
$$
\Delta (g)(y)=V_{f,g}(y)(g(\varphi_{f,g}(y)).
$$
Since $1=v^*(\Delta(g)(y))=v^*(V_{f,g}(y)(g(\varphi_{f,g}(y))))$, it follows that $\|g(\varphi_{f,g}(y))\|=1$, and hence $\psi(y)=\varphi_{f,g}(y)$. Therefore $$\Delta (f)(y)=V_{f,g}(y)(f(\varphi_{f,g}(y)))=V_{f,g}(y)(f(\psi(y)))=0.$$ We have established that
\begin{equation}\label{eq Deltafy vanishes when fpsiy does} \Delta (f)(y)=0, \hbox{ for every $f\in\Lip(X,E)$ with $f(\psi(y))=0$.}
\end{equation}
The previous arguments prove that $V(y)$ is well-defined.\smallskip

We shall show next that $V(y)$ is a surjective linear isometry on $E$ for each $y\in Y_0$. To this end, fix $y\in Y_0$ and let $u_0, v_0$ be elements in $E$. Let $f_0, g_0$ be arbitrary elements in $\Lip(X, E)$ attaining their norm only at $\psi(y)$ and satisfying that $f_0(\psi(y))=u_0$ and $g_0(\psi(y))=v_0$. It follows from the definition of $V(y)$ and the hypothesis on $\Delta$ that
$$
V(y)(u_0)=\Delta (f_0) (y)=V_{f_0,g_0}(y)(u_0)
$$
and
$$
V(y)(v_0)=\Delta (g_0)(y)=V_{f_0,g_0}(y)(v_0),
$$
where $V_{f_0,g_0}(y)\in \Iso(E)$. This shows that $V(y)$ is a $2$-local isometry on $E$. Since, by hypothesis, $E$ is $2$-iso-reflexive, we deduce that $V(y)$ is a surjective linear isometry on $E$.\smallskip

It is now time to explore the continuity of the mapping $V\colon Y_0\to\Iso(E)$. Suppose that $(y_n)$ is a sequence in $Y_0$ converging to $y_0\in Y_0$. For each $e\in E$, we have
$$
\|V(y_n)(e)-V(y_0)(e)\|=\|\Delta(1_X\otimes e)(y_n)-\Delta(1_X\otimes e)(y_0)\|$$ $$\leq\left\|\Delta(1_X\otimes e)\right\|_L d(y_n,y_0) = \left\|1_X\otimes e\right\|_L d(y_n,y_0)= \|e \|\ d(y_n,y_0)
$$											
for all $n\in\mathbb{N}$, and hence the mapping $V$ is continuous from $Y_0$ into $\Iso(E)$, when the latter is equipped with the SOT topology.\smallskip

We shall finally show that $\psi$ is a Lipschitz map. Indeed, fix any $e\in S_E$, and for any $y_1\neq y_2\in Y_0$, take $f_1=f_{\psi(y_1)}\otimes e$ and $f_2=f_{\psi(y_2)}\otimes e$. By hypothesis, we can find a surjective linear isometry $\Delta_{f_1,f_2}\colon\Lip(X,E)\to\Lip(Y,E)$ whose evaluations at $f_1$ and $f_2$ coincide with the corresponding evaluations of $\Delta$ at these two points. Therefore
$$
L(\Delta(f_1))\leq\|\Delta(f_1)\|_L=\|\Delta_{f_1, f_2}(f_1)\|_L$$ $$=\|f_1\|_L=\|f_{\psi(y_1)}\|_L\leq\max\left\{\diam(X),1\right\},
$$
and thus
$$
\|\Delta(f_1)(y_1)-\Delta(f_1)(y_2)\|\leq (\max\left\{\diam(X),1\right\})\ d(y_1, y_2).
$$
On the other hand, by \eqref{eq Deltafy vanishes when fpsiy does} we have $\Delta(f_1)(y_1)=0,$
and by the properties we established to show that $V$ is well defined, we deduce that
$$
\Delta(f_1)(y_2)=f_{\psi(y_1)}(\psi(y_2))\  V(y_2)(e) =d(\psi(y_1),\psi(y_2))\ V(y_2)(e).
$$
Since $V(y_2)$ is an isometry, we derive that
$$
d(\psi(y_1), \psi(y_2))=\|\Delta(f_1)(y_1)-\Delta(f_1)(y_2)\|\leq(\max\left\{\diam(X),1\right\})\ d(y_1,y_2).
$$
\end{proof}

Similar arguments to those given in the proof of Theorem \ref{thm:vL3} can be also applied to prove the following result.

\begin{theorem}\label{thm:vL3 AraujoDubarbie}
Let $X$ and $Y$ be compact metric spaces, and let $E$ be a smooth reflexive Banach space which is $2$-iso-reflexive. Let us assume that $\Delta\colon\Lip(X,E)\to\Lip(Y,E)$ is a $2$-local standard isometry.
Then there exist a subset $Y_0$ of $Y$, a Lipschitz bijection $\psi$ from $Y_0$ onto $X$, and a continuous mapping $V: Y_0\to (\Iso(E),\mathrm{SOT})$ such that
$$
\Delta(f)(y)=V(y)(f(\psi(y))),\quad\forall y\in Y_0, \ f\in\Lip(X,E).
$$ $\hfill\Box$
\end{theorem}

Our next technical result has been borrowed from \cite{JV11}.

\begin{lemma}\cite[Proposition 3.2]{JV11}\label{lem:vL4}
Let $X$ be a metric space, and let $R=\{r_n\colon  n\in \mathbb{N}\}$ be a countable set of pairwise distinct points of $X$.
Then there exist Lipschitz functions $f,g\colon  X\to [0, 1]$ such that $f$ has a strict local maximum at every point of $R$ and
$$
\{z\in X\colon  (f(z),g(z))=(f(r_n),g(r_n))\}=\{r_n\},\quad\forall\,n\in \mathbb{N}.
$$
\end{lemma}

We are now in position to prove our final results. Theorem \ref{thm:vL3} and Lemma \ref{lem:vL4} provide the tools to apply an standard method of proof introduced by Gy\H{o}ry in \cite{Gyo01}.

\begin{theorem}\label{thm:main1}
Let $X$ and $Y$ be pathwise connected compact metric spaces, and let $E$ be a smooth reflexive Banach space which is $2$-iso-reflexive.
Suppose that $\Delta$ is a $2$-local isometry from $\Lip(X,E)$ to $\Lip(Y,E)$.
Then there exist a Lipschitz homeomorphism $\psi$ from $Y$ onto $X$ and a continuous map $V$ from $Y$ into $(\Iso(E),\mathrm{SOT})$ such that
$$
\Delta f(y)=V(y)(f(\psi(y))),\quad\forall y\in Y,\ f\in\Lip(X,E).
$$
Moreover, $\Delta$ is a surjective linear isometry.
\end{theorem}

\begin{proof} Let $Y_0\subset Y$ and $\psi\colon Y_0\to X$ be the set and the Lipschitz bijection given by Theorem \ref{thm:vL3}. Since $\Delta$ is a $2$-local isometry, in particular, there exists a surjective linear isometry from $\Lip(X,E)$ to $\Lip(Y,E)$. It follows from Theorem \ref{thm BoFleJam} that there exists a Lipschitz homeomorphism $\psi_0\colon  Y\to X$. If $Y$ is finite, then $Y$, $X$ and $Y_0$ all have the same cardinality and thus $Y_0=Y$. So, we may assume that $Y$ is infinite.\smallskip

By the separability of $X$, we can find a countable dense subset $(x_n)_n$ of $X$. For each $n\in \mathbb{N},$ we set $y_n=\psi^{-1}(x_n)\in Y_0$. Suppose that $f$ and $g$ are the functions whose existence is guaranteed by Lemma \ref{lem:vL4}.
For a fixed $e\in S_E$, it follows from the $2$-local property of $\Delta$ that there exist a Lipschitz mapping $V_{f\otimes e, g\otimes e}\colon Y\to\Iso(E)$ and a Lipschitz homeomorphism $\varphi_{f\otimes e, g\otimes e}$ from $Y$ onto $X$ such that
$$
\Delta(f\otimes e)(y_n)=f(\varphi_{f\otimes e, g\otimes e}(y_n))\ V_{f\otimes e, g\otimes e}(y_n)(e),
$$
and
$$
\Delta(g\otimes e)(y_n)=g(\varphi_{f\otimes e, g\otimes e}(y_n))\ V_{f\otimes e, g\otimes e}(y_n)(e)
$$
for all $n\in \mathbb{N}$. Combining this identities with the fact that $y_n\in Y_0$ for every $n$, we deduce, via Theorem \ref{thm:vL3}, that
$$
\|\Delta(f\otimes e)(y_n)\|=f(\varphi_{f\otimes e,g\otimes e}(y_n))=\|f(\psi(y_n))V(y_n)(e)\|=\!f(\psi(y_n))\!=\!\!f(x_n),
$$
and
$$
\|\Delta(g\otimes e)(y_n)\|=g(\varphi_{f\otimes e,g\otimes e}(y_n))=\|g(\psi(y_n))V(y_n)(e)\|=\!g(\psi(y_n))\!=\!g(x_n),
$$
where $V\colon Y_0\to\Iso(E)$ is the mapping given by Theorem \ref{thm:vL3}. By the assumption on $f$ and $g$ in Lemma \ref{lem:vL4}, one can derive that $$x_n=\varphi_{f\otimes e, g\otimes e}(y_n)=\psi(y_n)
\hbox{ for all $n\in\mathbb{N}$.}$$

For any $y_0\in Y_0$, there exists a subsequence $(x_{n_k})_k$ of $(x_n)_n$ such that $(x_{n_k})_k\to \varphi_{f\otimes e, g\otimes e}(y_0)\in X$,
which implies that $(\varphi^{-1}_{f\otimes e, g\otimes e}(x_{n_k}))_k\to y_0$.
It follows from the equality $(\varphi^{-1}_{f\otimes e,g\otimes e}(x_{n_k}))_k=(\psi^{-1}(x_{n_k}))_k=(y_{n_k})_k$ that $(\psi^{-1}(x_{n_k}))_k\to y_0$ and hence $(x_{n_k})_k\to\psi(y_0)$. This shows that $\psi=\varphi_{f\otimes e,g\otimes e}$ on the whole $Y_0$.\smallskip

Finally, we prove that $Y_0=Y$. For any $y\in Y$, the element $\varphi_{f\otimes e,g\otimes e}(y)$ lies in $X=\psi(Y_0)$, and hence there exists  $y_0\in Y_0$ such that $\psi(y_0)=\varphi_{f\otimes e,g\otimes e}(y)$. This implies that
$$
\varphi_{f\otimes e,g\otimes e}(y_0)=\psi(y_0)=\varphi_{f\otimes e,g\otimes e}(y),
$$
and then $y=y_0\in Y_0$ because $\varphi_{f\otimes e, g\otimes e}$ is bijective.
\end{proof}

Thanks to Theorem \ref{thm:vL3 AraujoDubarbie}, the previous arguments also work in the case of $2$-local standard isometries.

\begin{theorem}\label{thm:main1 Araujo Dubarbie}
Let $X$ and $Y$ be compact metric spaces, and let $E$ be a smooth reflexive Banach space which is $2$-iso-reflexive. Suppose that $\Delta$ is a $2$-local standard isometry from $\Lip(X,E)$ to $\Lip(Y,E)$.
Then there exist a Lipschitz homeomorphism $\psi$ from $Y$ onto $X$ and a continuous map $V$ from $Y$ into $(\Iso(E),\mathrm{SOT})$ such that
$$
\Delta f(y)=V(y)(f(\psi(y))),\quad\forall y\in Y,\ f\in\Lip(X,E).
$$
Moreover, $\Delta$ is a surjective linear isometry.$\hfill\Box$
\end{theorem}

In Theorems \ref{thm:main1} and \ref{thm:main1 Araujo Dubarbie}, we have proved that, under certain assumptions on $X$ and $E$, the space $\Lip(X,E)$ is $2$-iso-reflexive and $2$-standard-iso-reflexive, respectively. One of the requirements is that $E$ is $2$-iso-reflexive. In fact, that condition is necessary to get the conclusion of these theorems. Actually, the same arguments given in the proof of \cite[Theorem 3]{AF09} can be applied to get the following result.

\begin{theorem}\label{final}
Let $X$ be a metric space and let $E$ be a Banach space. If $\Lip(X,E)$ is $2$-iso-reflexive or $2$-standard-iso-reflexive, then $E$ is $2$-iso-reflexive.$\hfill\Box$
\end{theorem}

\begin{remark}
Our theorems also hold if, in their hypotheses, we consider a $2$-local (standard) isometry $\Delta$ from $\Lip(X,E)$ to $\Lip(Y,F)$, where $E$ and $F$ are smooth reflexive Banach spaces for which every $2$-local isometry from $E$ to $F$ is both linear and surjective. In such a case, the map $V$ obtained in the conclusion should take its values in $\Iso(E,F)$.
\end{remark}

\textbf{Acknowledgements.} A. Jim{\'e}nez-Vargas and A.M. Peralta are partially supported by the Spanish Ministry of Economy and Competitiveness and European Regional Development Fund project no. MTM2014-58984-P and Junta de Andaluc\'{\i}a grants FQM194 and FQM375. L. Li is partly supported by NSF of China (11301285), L. Wang is partly supported by NSF of China Grants No. 11371222 and 11671133, and Y.-S. Wang is partly supported by Taiwan MOST 106-2115-M-005-001-MY2.\smallskip

The results presented in this note were originated during a visit of A.M. Peralta at the School of Mathematical Sciences in Nankai University. He would like to thank the second author and the Department of Mathematics for the hospitality during his stay.

\end{document}